\documentclass[12pt,twoside]{amsart}

\usepackage{amssymb,latexsym}

\usepackage{hyperref}
\hypersetup{
  colorlinks   = true, 
  urlcolor     = blue, 
  linkcolor    = blue, 
  citecolor   = red 
}
\usepackage{anysize}

\newtheorem{theorem}{Theorem}[section]

\newtheorem{lemma}[theorem]{Lemma}

\newtheorem{corollary}[theorem]{Corollary}

\theoremstyle{definition}

\theoremstyle{remark}

\numberwithin{equation}{section}

\renewcommand{\epsilon}{\varepsilon}
\renewcommand{\phi}{\varphi}
\renewcommand{\kappa}{\varkappa}

\begin{document}

\title[Estimating symplectic capacities from\dots]{Estimating symplectic capacities from lengths of closed curves on the unit spheres}

\author{Arseniy~Akopyan{$^\spadesuit$}}
\email{akopjan@gmail.com}


\author{Roman~Karasev{$^\diamondsuit$}}
\email{r\_n\_karasev@mail.ru}
\urladdr{http://www.rkarasev.ru/en/}

\address{{$^\spadesuit$} Institute of Science and Technology Austria (IST Austria), Am Campus 1, 3400 Klosterneuburg, Austria}
\address{
{$^\diamondsuit$}Moscow Institute of Physics and Technology, Institutskiy per. 9, Dolgoprudny, Russia 141700}
\address{{$^\diamondsuit$}Institute for Information Transmission Problems RAS, Bolshoy Karetny per. 19, Moscow, Russia 127994}

\thanks{{$^\diamondsuit$} Supported by the Russian Foundation for Basic Research Grant 18-01-00036}


\begin{abstract}
We improve the estimates for the Ekeland--Hofer--Zehnder capacity of convex bodies by Gluskin and Ostrover. In the course of our argument we show that a closed characteristic of minimal action on the boundary of a centrally symmetric convex body in $\mathbb R^{2n}$ must itself be centrally symmetric, with generalizations to some other types of symmetry.
\end{abstract}

\maketitle

\section{Introduction}

We are going to improve some results of the recent paper~\cite{gluskin2015} on estimating the symplectic capacities of convex bodies in $\mathbb R^{2n}$. The symplectic capacities are interesting invariants of open domains in the symplectic $\mathbb R^{2n}$, which are in most cases are not easy to estimate. Some particular results, including sharp estimates, can be found in \cite{aao2008,aao2012,aaok2013,abksh2014}.

Following the conventions in~\cite{gluskin2015}, we identify $\mathbb R^{2n} = \mathbb C^n$, that is, besides the symplectic structure $\omega = d\lambda$, we are going to use the complex structure $J : \mathbb C^n\to \mathbb C^n$.

The values to estimate are the Ekeland--Hofer--Zehnder capacity of a convex body $K\subset \mathbb C^n$, denoted by $c_{EHZ}(K)$. This capacity is known to coincide (for smooth convex bodies) with the smallest action
$$
A(\gamma) = \int_0^T \gamma^* \lambda\; dt
$$
of a \emph{closed characteristic} $\gamma : [0, T] \to \partial K$, $\gamma(0)=\gamma(T)$, that is a closed curve whose tangent vector belongs to the kernel of the restriction of $\omega$ to $\partial K$. 

The key idea of~\cite{gluskin2015} is to introduce another invariant, let us denote it by $c_J(K)$, which is the inverse of the norm of $J$ considered as a linear operator acting from the space $\mathbb C$ with the unit gauge $K^\circ$ (the polar of $K$) to the same space $\mathbb C$ with the unit gauge $K$. Equivalently,
\[
c_J(K)^{-1} = \max \{\omega(x,y) : x,y\in K^\circ\}.
\] 
Here we use the technical assumption that the origin is in the interior of $K$ in order to be able to work with its polar $K^\circ$; and we say \emph{gauge} to insist that the body $K$ need not be centrally symmetric, although the particular case of centrally symmetric body will be of interest to us. We denote the corresponding \emph{gauge function} (not necessarily symmetric norm) by $g_K(\cdot)$, which is the $1$-homogeneous function that is equal to $1$ precisely on $\partial K$. The value $c_J(K)$ is not a symplectic invariant (under arbitrary symplectomorphisms), but it is invariant under linear symplectomorphisms.

Proposition 2.2 of~\cite{gluskin2015} established the estimate
$$
c_{EHZ}(K) \ge c_J(K),
$$
and we are going to improve it in the following way:

\begin{theorem}
\label{theorem:ehz-estimate}
For arbitrary convex body $K\subset \mathbb C^n$,
$$
c_{EHZ}(K) \ge \left(1 + \frac{1}{2n}\right) c_J(K);
$$
and for centrally symmetric convex body $K\subset \mathbb C^n$
$$
c_{EHZ}(K) \ge \left(2 + \frac{1}{n}\right) c_J(K).
$$
\end{theorem}

It seems that these improved bounds are still not tight and can be further improved, in particular, because the equality case in Lemma \ref{lemma:schaffer} below does not seems to correspond to a closed characteristic on the boundary of $K$ under a symplectic structure on $\mathbb R^d$.

\subsection*{Acknowledgments.} 
The authors thank Yaron Ostrover for his interest in the results and useful remarks.

\section{Length estimates for curves on the boundary}

In order to prove the main theorem we are going to use the approach of Bezdek and Bezdek \cite{bb2009} and estimate the lengths of closed curves $\gamma\in\mathbb R^d$ that cannot be translated into the interior of $K$. We measure their lengths with $g_K$ and note that in the symplectic case closed characteristics on the boundary of $K$ cannot be translated into the interior of $K$, this is explained in Section~\ref{section:proofs}.

In \cite{abksh2014} such estimates for a convex body $K\subset\mathbb R^n$ were used to produce lower bounds for the Ekeland--Hofer--Zehnder capacity of \emph{Lagrangian product} $K\times K^\circ\subset\mathbb R^{2n}$. This time we are going to use such estimates for any convex body $K$ in the symplectic $\mathbb R^{2n}$.

In the question about closed characteristics, compared to the Bezdek--Bezdek setup, we have an additional property of the curve $\gamma$: it is contained in the boundary of $K$. Now we want to check if this improves the lower bound on the $g_K$-length of such a curve, if it does not fit into the interior of a translate of $K$. We will use the general estimate $\ell(\gamma)\ge 2+2/d$ from \cite[Theorem 4.1]{abksh2014} for not necessarily symmetric $K$, although we believe that this estimate is not tight. And we are going to consider a symmetric $K$ separately and improve the bound $\ell(\gamma)\ge 4$ of \cite[Theorem 3.2]{abksh2014} in the symmetric case using the fact that the curve lies on the boundary.

It is well known (see \cite{reshetnyak1953anextremal}) that in dimension $d=2$ (when we speak about $\mathbb R^d$ without symplectic structure we just denote its dimension by $d$) the length of the boundary of the unit disk of a normed plane is at least $6$ in this norm; so the bound $4$ that we mentioned above is not optimal in this dimension. But in dimensions $d\ge 3$ we may consider the cube $C = [-1, 1]^d$ and a closed curve that is just an edge of $C$ going forth and back. This closed curve on the boundary has length $4$ in the norm $g_C(\cdot)$ and does not fit into the interior of a translate of $C$; so the bound $4$ cannot be improved without additional assumptions on the curve, which we are going to establish below in the case of action-minimizing closed characteristics.

\subsection{Action minimizers on the boundary of centrally symmetric convex bodies}

Of course, we have not used the symplectic structure in any way, but so far we see no evidence why a closed characteristic (under some symplectic structure of $\mathbb R^d=\mathbb R^{2n}$) on the boundary of a smooth convex body cannot behave like in the above example. We are going to analyze the behavior of a closed characteristic on the boundary of a centrally symmetric convex $K\subseteq\mathbb R^{2n}$ having the smallest action and giving the value of $c_{EHZ}(K)$. We will pass from convex bodies to norms $\|\cdot\|$ (as support functions of $K$ composed with $\omega$) and consider the classical (see~\cite{clarke1979}) variational problem for closed loops $\gamma : \mathbb R/\mathbb Z \to \mathbb R^{2n}$:

\begin{equation}
\label{equation:clarke}
\int_\gamma \|\dot\gamma\| \to \min, \quad \int_\gamma \lambda = 1,
\end{equation}
where $\lambda$ is a primitive of $\omega$. The minimum in this variational problem is the Ekeland--Hofer--Zehnder capacity $c_{EHZ}(K)$ of the convex domain $X$ up to constant, as shown in~\cite{clarke1979,ekeland1989}. Now we prove:

\begin{lemma}
\label{lemma:symmetry}
In the variational problem \eqref{equation:clarke}, for centrally symmetric $\|\cdot\|$, one of the minima is attained at a curve $\gamma$ centrally symmetric with respect to the origin. For smooth and strictly convex norms we can say more: All minima of \eqref{equation:clarke} are centrally symmetric with respect to some center.
\end{lemma}

In~\cite[Corollary 2]{clarke1982} the existence of some symmetric closed characteristics was proved, but here we need more, that a minimal symmetric closed characteristic has no greater action than the minimal action of an arbitrary closed characteristic.

\begin{proof}
Assume $\gamma$ is the minimal curve, split it into two curves $\gamma_1$ and $\gamma_2$ of equal $\|\cdot\|$-lengths. We are going to use the particular primitive $\lambda = \sum_i p_idq_i$, which is invariant under the central symmetry of $\mathbb R^{2n}$.

Since the problem is invariant under translations of $\gamma$, we may translate it so that $\gamma_1$ passes from $-x$ to $x$ and $\gamma_2$ passes from $x$ to $-x$. Let $\sigma$ be the straight line segment from $x$ to $-x$ and $-\sigma$ be its opposite. Then the concatenations $\beta_1 = \gamma_1\cup\sigma$ and $\beta_2 = \gamma_2\cup (-\sigma)$ are the closed loops such that
\[
\int_{\beta_1} \lambda + \int_{\beta_2} \lambda = \int_{\gamma} \lambda = 1.
\]
Then without loss of generality we assume $\int_{\beta_1} \lambda \ge 1/2$. Then the centrally symmetric curve $\gamma' = \gamma_1\cup(-\gamma_1)$ has
\[
\int_{\gamma'} \lambda = \int_{\beta_1} \lambda + \int_{-\beta_1} \lambda \ge 1,
\]
and the length of $\gamma'$ is the same as the length of $\gamma$. Scaling $\gamma'$ to obtain $\gamma''$ with $\int_{\gamma''} \lambda = 1$ we will have centrally symmetric $\gamma''$ with length no greater than the length of the original $\gamma$.

In case the original body $K$ is smooth and strictly convex (and so is the norm $\|\cdot\|$) the proof shows that any minimal closed curve must be centrally symmetric. Indeed, we have presented a process that symmetrizes the minimal curve keeping a half of it the same. From the standard facts about the variational problems with everything smooth it follows that the minimum of the variational problem must be a smooth solution to a certain ODE, so the symmetrized curve must actually remain the same by the uniqueness of the solutions to ODEs. Therefore it must already be centrally symmetric.
\end{proof}

Using the correspondence \cite{clarke1979} between the optimization problem \eqref{equation:clarke} and the closed characteristics on the boundary of a smooth and strictly convex $K$ we may rephrase the previous lemma:

\begin{corollary}
\label{corollary:symmetry}
If $K\subset\mathbb R^{2n}$ is a smooth and strictly convex centrally symmetric body then any closed characteristic of minimal action on the boundary $\partial K$ is itself centrally symmetric.
\end{corollary}
\begin{proof}
In \cite{clarke1979} the problem \eqref{equation:clarke} was actually stated not for the length, but for the integral of $\|\dot\gamma\|^2$, this is needed to ensure an appropriate parameterization of $\gamma$, while the length version is invariant under reparameterizing $\gamma$ and therefore has some indeterminacy.

Lemma~\ref{lemma:symmetry} ensures that the optimal $\gamma$ is centrally symmetric after a translation; and we need to check that the corresponding closed characteristic in the boundary of $K$ is also centrally symmetric. In \cite[Page 187, 3rd line from the bottom]{clarke1979} the curve $\gamma$ is translated (note that the problem \eqref{equation:clarke} is translation-invariant), but after this translation $\gamma$ is expressed as the gradient of $\|\dot\gamma \|^2$ composed with the complex structure $J$. Since the derivative $\dot\gamma$ remains centrally symmetric under translations of $\gamma$, therefore the resulting translation of $\gamma$ in \cite[Page 187, 3rd line from the bottom]{clarke1979} is also centrally symmetric. The rest of the argument in \cite{clarke1979} shows that after such transformations $\gamma$ becomes a closed characteristic on the boundary of a homothet of $K$.
\end{proof}

In fact, it is possible to extend the results of Lemma~\ref{lemma:symmetry} and Corollary~\ref{corollary:symmetry} to other classes of symmetry:

\begin{theorem}
Let $K\subset\mathbb C^n$ be a smooth and strictly convex body such that $wK = K$ for the primitive $m$th root of unity $w\in\mathbb C$. Then any closed characteristic of minimal action on the boundary $\partial K$ is invariant under the multiplication by $w$.
\end{theorem}

Compare this theorem with the result of \cite{clarke1982} showing (as a particular case) the existence of such closed geodesics without proving their minimality.

\begin{proof}
Like above we consider Clarke's optimization problem \eqref{equation:clarke} with the (not necessarily symmetric) norm $\|\cdot\|$ invariant under the multiplication by $w$. Let $\gamma$ be a solution to this problem. Split $\gamma$ into $m$ segments $\gamma_1,\ldots,\gamma_m$ of equal $\|\cdot\|$-length and let for any $i=1,\ldots,m$ the vector $v_i$ go from the beginning of $\gamma_i$ to its end. The action of $\gamma$ can be expressed as the sum
\[
A(\gamma) = \int_\gamma \lambda = A + A_1 + \dots + A_m,
\] 
where $A$ is the action of the polygon composed of the vectors $v_i$ and $A_i$ is the action of the loop obtained by concatenation of $\gamma_i$ and $-v_i$.

Now we try to replace $\gamma$ with the closed curve made by concatenation of appropriately translated $\gamma_i, w\gamma_i$, \dots, $w^{m-1}\gamma_i$. We indeed can arrange them into a closed loop because we have
\[
1 + w + \dots + w^{m-1}=0.
\]
This new loop, after an appropriate translation, will be invariant under the multiplication by $w$. The action of this new loop will be
\[
m A_i + S_i,
\]
where $S_i$ is the Riemannian area of the regular $m$-gon with side vectors $v_i, wv_i,\ldots, w^{m-1}v_i$ inscribed into this loop. This $m$-gon lies in a single complex line, which is a real two-dimensional subspace where the Riemannian area coincides with the symplectic action of its contour.

In fact, there is a constant (we do not need its precise value) $\alpha_m$ such that the area of a regular $m$-gon is expressed through its side length $\ell$ as $\alpha_m\ell^2$. Hence $S_i = \alpha_m |v_i|^2$, here we use the Euclidean norm, not $\|\cdot\|$. Now assume that this attempt to build a no worse $w$-invariant solution to Clarke's variational problem fails and we have
\[
m A_i + \alpha_m |v_i|^2 < A + \sum_i A_i.
\]
Moreover, assume this attempt fails for all $i=1,\ldots,m$. Then summing such inequalities we obtain
\begin{equation}
\label{equation:square-isoperimetry}
\alpha_m \sum_i |v_i|^2 < m A,
\end{equation}
and from the quadratic-arithmetic mean inequality
\[
\frac{\alpha_m}{m}\left( \sum_i |v_i| \right)^2 \le \alpha_m \sum_i |v_i|^2 < m A.
\]
Note that we can fill the polygonal loop composed of $v_1,\ldots, v_m$ by a $2$-chain $F$ of triangles $T_1,\ldots, T_{m-2}$ and note that by the calibration property \cite[Section I]{harvey-lawson1982} the Riemannian area of every such $T_i$ is no less than its symplectic action, hence for the area $S_F$ of the whole $2$-chain $F$ we have
\[
\frac{\alpha_m}{m}\left( \sum_i |v_i| \right)^2  < m S_F.
\]
This looks like the inverse isoperimetric inequality. Moreover, we can rotate $T_2,T_3,\ldots, T_{m-2}$ into the same $2$-plane one by one and see that this inequality is the opposite to the standard isoperimetric inequality for $m$-gons in the plane, thus obtaining a contradiction.

Therefore, for some $i$ the attempt to produce no worse solution with the required symmetry must succeed. From the existence and uniqueness theorem for the ODE for this variational problem we see that in the smooth case this modification of $\gamma$ must in fact give the same loop, just because the segment $\gamma_i$ remains at its place.

\textit{An alternative way to deduce \eqref{equation:square-isoperimetry}:} Note that the action of a loop is the sum of actions of its projections to the coordinate complex lines. At the same time the Euclidean square of a vector is the sum of the Euclidean squares of its projection to the same complex lines. Therefore it is sufficient to prove the opposite to \eqref{equation:square-isoperimetry} in the plane, where it again follows from the standard isoperimetry for $m$-gons.
\end{proof}

\subsection{Length estimates for centrally symmetric curves on the spheres}

In view of the previous symmetry results we need to investigate the length of centrally symmetric curves on the boundary of $K$ measured in the norm with unit ball $K\subset\mathbb R^d$. 

The following result is essentially \cite[Theorems 13E and 13F]{schaffer1976}:

\begin{lemma}[Schaffer, 1976]
\label{lemma:schaffer}
If $K\subset\mathbb R^d$ is convex and centrally symmetric and $\gamma\subset \partial K$ is a centrally symmetric closed curve then the length of $\gamma$ in the norm whose unit ball is $K$ is at least $4 + 4/d$. For odd $d$ the bound can be improved to $4+4/(d-1)$.
\end{lemma}

The example in \cite[Theorem 10E]{schaffer1976} shows that this bound is tight for even $d$, in the case if interest to us.

\section{Proof of Theorem~\ref{theorem:ehz-estimate}}
\label{section:proofs}

Now we turn to the estimates of the Ekeland--Hofer--Zehnder capacity. Following~\cite{gluskin2015} we consider a closed characteristic $\gamma : [0, T] \to \partial K$ as given by the equation:
\begin{equation}
\label{equation:characteristic}
\dot\gamma = J\nabla g_K(\gamma);
\end{equation}
in this case the action $A(\gamma)$ equals $T/2$. Integrating the equation (\ref{equation:characteristic}) we see that the gradients $\nabla g_K$ (that is the outer normals to $\partial K$) along the curve $\gamma$ non-negatively combine to give zero. This implies (see~\cite{abksh2014} or \cite{aaok2013}) that the curve $\gamma$ cannot be translated to the interior of $K$. 

Now in the not necessarily symmetric case we just invoke the estimates~\cite[Theorem 4.1]{abksh2014} (going back to~\cite{bb2009}) for such closed curves that cannot be covered by a smaller homothets of $K$:
$$
\int_0^T g_K(\dot\gamma)\; dt \ge 2+\frac{1}{n}.
$$

In the centrally symmetric case we have
$$
\int_0^T g_K(\dot\gamma)\; dt \ge 4 + \frac{2}{n}
$$
in view of Lemma \ref{lemma:schaffer}, because in the case of centrally symmetric $K$ we have a centrally symmetric curve by Lemma \ref{lemma:symmetry} (or Corollary~\ref{corollary:symmetry}).

We conclude the proof as in~\cite{gluskin2015}:
$$
\int_0^T g_K(\dot\gamma)\; dt \le \int_0^T g_K(J\nabla g_K(\gamma))\; dt \le \int_0^T \|J\|_{K^\circ\to K}\; dt = \frac{T}{c_J(K)} = \frac{2 A(\gamma)}{c_J(K)},
$$
which gives the required estimates.

\bibliography{../Bib/karasev}

\begin{thebibliography}{10}

\bibitem{abksh2014}
A.~Akopyan, A.~Balitskiy, R.~Karasev, and A.~Sharipova.
\newblock Elementary approach to closed billiard trajectories in asymmetric
  normed spaces.
\newblock {\em Proceedings of the American Mathematical Society},
  144(10):4501--4513, 2016.
\newblock \href{http://arxiv.org/abs/1401.0442}{arXiv:1401.0442}.

\bibitem{aaok2013}
S.~Artstein-Avidan, R.~Karasev, and Y.~Ostrover.
\newblock From symplectic measurements to the {M}ahler conjecture.
\newblock {\em Duke Mathematical Journal}, 163(11):2003--2022, 2014.
\newblock \href{http://arxiv.org/abs/1303.4197}{arXiv:1303.4197}.

\bibitem{aao2008}
S.~Artstein-Avidan and Y.~Ostrover.
\newblock A {B}runn--{M}inkowski inequality for symplectic capacities of convex
  domains.
\newblock {\em International Mathematics Research Notices}, 2008.
\newblock \href{http://arxiv.org/abs/0712.2631}{arXiv:0712.2631}.

\bibitem{aao2012}
S.~Artstein-Avidan and Y.~Ostrover.
\newblock Bounds for {M}inkowski billiard trajectories in convex bodies.
\newblock {\em International Mathematics Research Notices}, 2012.
\newblock \href{http://arxiv.org/abs/1111.2353}{arXiv:1111.2353}.

\bibitem{bb2009}
D.~Bezdek and K.~Bezdek.
\newblock Shortest billiard trajectories.
\newblock {\em Geometriae Dedicata}, 141:197--206, 2009.

\bibitem{clarke1979}
F.~Clarke.
\newblock A classical variational principle for periodic {H}amiltonian
  trajectories.
\newblock {\em Proceedings of the American Mathematical Society}, 76:186--188,
  1979.

\bibitem{clarke1982}
F.~Clarke.
\newblock On {H}amiltonian flows and symplectic transformations.
\newblock {\em SIAM Journal of Control and Optimization}, 20(3):355--359, 1982.

\bibitem{ekeland1989}
I.~Ekeland and H.~Hofer.
\newblock Symplectic topology and {H}amiltonian dynamics.
\newblock {\em Mathematische Zeitschrift}, 200(3):355--378, 1989.

\bibitem{gluskin2015}
E.~D. Gluskin and Y.~Ostrover.
\newblock Asymptotic equivalence of symplectic capacities.
\newblock 2015.
\newblock \href{http://arxiv.org/abs/1509.01797}{arXiv:1509.01797}.

\bibitem{harvey-lawson1982}
R.~Harvey and H.~B. Lawson~Jr.
\newblock Calibrated geometries.
\newblock {\em Acta Mathematica}, 148(1):47--157, 1982.

\bibitem{reshetnyak1953anextremal}
Y.~G. Reshetnyak.
\newblock An extremal problem from the theory of convex curves.
\newblock {\em Uspehi Matem. Nauk (N.S.)}, 8(6(58)):125--126, 1953.

\bibitem{schaffer1976}
J.~J. Schaffer.
\newblock {\em Geometry of Spheres in Normed Spaces}.
\newblock Lecture Notes in Pure and Applied Mathematics. Marcel Dekker Inc,
  1976.

\end{thebibliography}
\bibliographystyle{abbrv}
\end{document}